\documentclass[reqno]{article}
\usepackage{hyperref}
\usepackage{amsthm}
\usepackage{graphicx}
\usepackage{amsfonts}
\usepackage{amsmath}
\usepackage{amssymb}
\usepackage{mathrsfs}
\usepackage[margin=1in]{geometry}
\usepackage{mathtools}
\usepackage{multicol}
\usepackage{pictexwd,dcpic}
\theoremstyle{definition}
\newtheorem{theorem}{Theorem}
\newtheorem{lemma}[theorem]{Lemma}
\newtheorem{example}[theorem]{Example}

\newtheorem{corollary}[theorem]{Corollary}

\newtheorem{definition}[theorem]{Definition}
\def\N{\mathbb{N}}

\def\L{\mathscr{L}}
\def\Lom{\L_{\epom}}

\def\M{\mathscr{M}}

\def\FV{\mathrm{FV}}
\def\LPA{\L_{\mathrm{PA}}}
\def\LEA{\L_{\mathrm{EA}}}

\def\LEA{\L_{\mathrm{EA}}}

\def\epom{\omega\cdot\omega}

\def\onset{\mathrm{On}}

\def\depth{\mathrm{depth}}

\newcommand{\ucl}[1]{\mathrm{ucl}(#1)}

\newcommand{\case}[1]{\textbf{Case #1:}}

\raggedcolumns

\begin{document}

\title{Fast-collapsing theories}
\author{Samuel A.~Alexander\thanks{Email:
alexander@math.ohio-state.edu}\\
\emph{Department of Mathematics, the Ohio State University}}
\date{February 2015}
\maketitle

\begin{abstract}
Reinhardt's conjecture, a formalization of the statement
that a truthful knowing machine can know its own truthfulness
and mechanicalness, was proved by Carlson
using sophisticated structural results about the ordinals
and transfinite induction just beyond the first epsilon number.
We prove a weaker version of the conjecture, by elementary
methods and transfinite induction up to a smaller ordinal.
\end{abstract}

\section{Introduction}

This is a paper about idealized truthful mechanical knowing agents
who know facts in a quantified arithmetic-based language that also includes a
connective for their own knowledge ($K(1+1=2)$ is read ``I (the agent)
know $1+1=2$'').
It is well known (\cite{benacerraf},
\cite{carlson2000}, \cite{lucas}, \cite{penrose}, \cite{putnam}, \cite{reinhardt})
that such an agent cannot simultaneously know its own truthfulness and its own code.
Reinhardt conjectured that,
while knowing its own truthfulness, such a machine can
know it has \emph{some} code, without knowing which.
This conjecture was proved by Carlson \cite{carlson2000}.
The proof uses sophisticated structural results from \cite{carlson1999} about the ordinals,
and involves transfinite induction up to $\epsilon_0\cdot\omega$.

We will give a proof of a weaker result, but will do so in an
elementary way, inducting only as far as $\epom$.
Along the way, we will develop some machinery that is interesting in its own right.
Carlson's proof of Reinhardt's conjecture is based on stratifying knowledge
(see \cite{carlson2012} for a gentle summary).
This can be viewed as adding operators $K^\alpha$ for knowledge
after time $\alpha$ where $\alpha$ takes ordinal values.
Under certain assumptions, theories in such stratified language \emph{collapse}
at positive integer multiples of $\epsilon_0$, in the sense that if $\phi$
only contains superscripts $<\epsilon_0\cdot n$ ($n$ a positive integer)
then $K^{\epsilon_0\cdot n}\phi$ holds if and only if $K^{\epsilon_0\cdot(n+1)}\phi$ does.
In this paper, collapse occurs at positive integer multiples of $\omega$, hence the name:
\emph{Fast-collapsing theories}.

Our result is weakened in the sense that the background theory of knowledge is weakened.
The schema $K(\ucl{K(\phi\rightarrow\psi)\rightarrow K\phi\rightarrow K\psi})$ ($\mathrm{ucl}$ denotes 
universal closure)
is weakened by adding the requirement that $K$ not be nested deeper in $\phi$ than in $\psi$
(the unrestricted schema $\ucl{K(\phi\rightarrow \psi)\rightarrow K\phi\rightarrow K\psi}$
is preserved, but the knower is not required to \emph{know} it);
the schema $\ucl{K\phi\rightarrow KK\phi}$ is forfeited entirely;
and a technical axiom called Assigned Validity (made up of valid formulas with numerals
plugged in to their free variables) is added to the background theory of knowledge.

On the bright side, our result is stated in a more general way (we mention in passing
how the full unweakened result could also be so generalized, but leave those details for later work).
Casually, our main theorem has the following form:
\begin{quote}
A truthful knowing agent whose knowledge is sufficiently ``generic''
can be taught its own truthfulness and still remain truthful.
\end{quote}
Here ``generic'' is a specific technical term, but it is
inclusive enough to include knowledge that one has some code,
thus the statement addresses Reinhardt's conjecture.

In Section \ref{prelimsect} we present some preliminaries.

In Section \ref{stratifierssectn} we develop \emph{stratifiers}, maps from unstratified language to 
stratified language.
These are the key to fast collapse.  They debuted in 
\cite{alexanderdissert} and \cite{alexanderjsl}.

In Section \ref{uniformsect} we discuss \emph{uniform} stratified theories.
A key advantage of stratifiers is that
they turn unstratified theories into uniform stratified theories.

In Section \ref{genericaxiomssectn} we define some notions of genericity of an
axiom schema, and establish the genericity of some building blocks of background
theories of knowledge.

In Section \ref{mainresultsect} we state our main theorem and make closing remarks.

\section{Preliminaries}
\label{prelimsect}

\begin{definition}
\label{standarddefns}
(Standard Definitions)
Let $\LPA$ be the language $(0,S,+,\cdot)$ of Peano arithmetic and let $\L$ be an arbitrary language.
\begin{enumerate}
\item For any $e\in\N$, $W_e$ is the range of the $e$th partial computable function.
The binary predicate $\bullet\in W_\bullet$ is $\LPA$-definable so we will freely act
as if $\LPA$ actually contains this predicate symbol.
\item If an $\L$-structure $\M$ is clear from context, an \emph{assignment}
is a function taking variables into the universe of $\M$.
\item If $s$ is an assignment, $x$ is a variable, and $a\in\M$, $s(x|a)$ is the assignment
that agrees with $s$ except that $s(x|a)(x)=a$.
\item
We define $\LPA$-terms $\overline n$ ($n\in\N$), called \emph{numerals}, so that $\overline 0=0$
and
$\overline{n+1}=S(\overline n)$.
\item If $\phi$ is an $\L$-formula, $\FV(\phi)$ is the set of free variables of $\phi$.
If $\FV(\phi)=\emptyset$ then $\phi$ is a \emph{sentence}.
\item If $\phi$ is an $\L$-formula, $x$ is variable, and $u$ is an $\L$-term,
$\phi(x|u)$ is the result of substituting $u$ for all free occurrences of $x$ in $\phi$.
\item A \emph{universal closure} of an $\L$-formula $\phi$ is a sentence $\forall x_1\cdots \forall x_n\phi$.
We write $\ucl{\phi}$ to denote a universal closure of $\phi$.
\item We use the word \emph{theory} as synonym for \emph{set of sentences}.
\item If $T$ is an $\L$-theory and $\M$ is an $\L$-structure, $\M\models T$ means that $\M\models\phi$ for all $\phi\in T$.
\item If $T$ is an $\L$-theory, we say $T\models\phi$ if $\M\models\phi$ whenever $\M\models T$.
\item A \emph{valid} $\L$-formula is one that holds in every $\L$-structure.
\item For any formulas $\phi_1,\phi_2,\phi_3$, we write $\phi_1\rightarrow\phi_2\rightarrow\phi_3$ to abbreviate
$\phi_1\rightarrow(\phi_2\rightarrow\phi_3)$.
\end{enumerate}
\end{definition}

We will repeatedly use the following standard fact
without explicit mention: if $\psi$ is a universal closure of $\phi$,
then in order to prove $\M\models\psi$, it suffices to let $s$ be an arbitrary assignment and 
show that
$\M\models\phi[s]$.

For quantified semantics we work in Carlson's \emph{base logic}, defined as follows.

\begin{definition}
\label{baselogicdefn}
(The Base Logic)
A \emph{language $\L$ in the base logic} is a first-order language $\L_0$ together with a set
of symbols called \emph{operators}.
Formulas of $\L$ are defined in the usual way, with the clause that whenever $\phi$ is an $\L$-formula
and $K$ is an $\L$-operator, $K\phi$ is also an $\L$-formula (and $\FV(K\phi)=\FV(\phi)$).
Syntactic parts of Definition \ref{standarddefns} extend to the base logic in obvious ways.
Given such an $\L$, an \emph{$\L$-structure} $\M$ is a first-order $\L_0$-structure $\M_0$
together with a function that takes one $\L$-formula $\phi$, one $\L$-operator $K$, and one assignment $s$,
and outputs True or False---in which case we write $\M\models K\phi[s]$ or $\M\not\models K\phi[s]$, respectively---satisfying
the following three conditions (where $\phi$ ranges over $\L$-formulas and $K$ ranges over operators):
\begin{enumerate}
\item Whether or not $\M\models K\phi[s]$ is independent of $s(x)$ if $x\not\in\FV(\phi)$.
\item (Alphabetic Invariance) If $\psi$ is an \emph{alphabetic variant} of $\phi$, meaning that it is obtained from $\phi$ by renaming bound
variables while respecting binding of the quantifiers, then $\M\models K(\phi)[s]$ if and only if $\M\models K(\psi)[s]$.
\item (Weak Substitution)\footnote{Note that the general
substitution law, where $y$ is replaced by an arbitrary term,
is not valid in modal logic.} If the variable $y$ is substitutable for the variable $x$ in $\phi$,
then $\M\models K\phi(x|y)[s]$ if and only if $\M\models K\phi[s(x|s(y))]$.
\end{enumerate}
\end{definition}

\begin{theorem}
\label{completenesscompactness}
(Completeness and compactness)
Let $\L$ be an r.e.~language in the base logic.
\begin{enumerate}
\item The set of valid $\L$-formulas is r.e.
\item For any r.e.~$\L$-theory $T$, $\{\phi\,:\,T\models\phi\}$ is r.e.
\item There is an effective algorithm, given (a G\"odel number for) an r.e.~$\L$-theory $T$, to find
(a G\"odel number for) $\{\phi\,:\,T\models\phi\}$.
\item If $T$ is an $\L$-theory and $T\models\phi$ ($\phi$ any $\L$-formula), there
are $\tau_1,\ldots,\tau_n\in T$ such that $\left(\bigwedge_i\tau_i\right)\rightarrow\phi$ is valid.
\end{enumerate}
\end{theorem}

\begin{proof}
By interpreting the base logic in first-order logic.  For details,
see \cite{alexanderdissert}.
\end{proof}

\begin{definition}
Let $\LEA$ be the language of Epistemic Arithmetic from 
\cite{shapiro1985}, so $\LEA$
extends $\LPA$ by a unary operator $K$.
An $\LEA$-structure (more generally an $\L$-structure
where $\L$ extends $\LPA$) has \emph{standard first-order part} if its first-order part has universe $\N$ and
interprets $0,S,+,\cdot$ in the intended ways.
\end{definition}

\begin{definition}
Suppose $\L$ extends $\LPA$ and $\phi$ is an $\L$-formula with $\FV(\phi)\subseteq\{x_1,\ldots,x_n\}$.
For any assignment $s$ into $\N$,
we define
\[
\phi^s\equiv \phi(x_1|\overline{s(x_1)})\cdots (x_n|\overline{s(x_n)}),
\]
the sentence obtained by replacing all free variables in $\phi$ by numerals according to $s$.
\end{definition}

\begin{definition}
\label{defnofintendedmodel}
For any $\LEA$-theory $T$,
the intended structure for $T$ is the $\LEA$-structure $\mathscr N_T$
that has standard first-order part and interprets $K$ so that for any $\LEA$-formula $\phi$ and assignment $s$,
\[
\mbox{$\mathscr N_T\models K\phi[s]$ if and only if $T\models\phi^s$.}
\]
We say $T$ is \emph{true} if $\mathscr N_T\models T$.
\end{definition}

It is easy to check that the structures $\mathscr N_T$ of Definition \ref{defnofintendedmodel}
really are $\LEA$-structures (they satisfy Conditions 1--3 of Definition \ref{baselogicdefn}).
The following lemma shows that they accurately interpret
quantified formulas in the way one would expect.

\begin{lemma}
For any $\LEA$-theory $T$, $\LEA$-formula $\phi$ and assignment $s$,
\[
\mbox{$\mathscr N_T\models\phi[s]$ if and only if $\mathscr N_T\models \phi^s$.}
\]
\end{lemma}

\begin{proof}
Straightforward induction.
\end{proof}

Armed with these definitions, we can make more precise some things
we suggested in the introduction.
Let $T_{\text{SMT}}$ be the following $\LEA$-theory
($\phi$ and $\psi$ range over $\LEA$-formulas):
\begin{enumerate}
\item ($E_1$) $\ucl{K\phi}$ whenever $\phi$ is valid.
\item ($E_2$) $\ucl{K(\phi\rightarrow\psi)\rightarrow K\phi\rightarrow K\psi}$.
\item ($E_3$) $\ucl{K\phi\rightarrow\phi}$.
\item ($E_4$) $\ucl{K\phi\rightarrow KK\phi}$.
\item The \emph{axioms of Epistemic Arithmetic}, by which we mean the axioms of Peano Arithmetic
with the induction schema extended to $\LEA$.
\item (Mechanicalness) $\ucl{\exists e \forall x(K\phi\leftrightarrow x\in W_e)}$
provided $e\not\in\FV(\phi)$.
\item $K\phi$ whenever $\phi$ is an instance of lines 1--6 or (recursively) 7.
\end{enumerate}
Combining lines 6 and 7 yields the \emph{Strong Mechanistic Thesis},
$K(\ucl{\exists e \forall x(K\phi\leftrightarrow x\in W_e)})$.
One of the main results of \cite{carlson2000} is that $T_{\text{SMT}}$ is true, that is, $\mathscr N_{T_{\text{SMT}}}\models T_{\text{SMT}}$.
To establish $\mathscr N_{T_{\text{SMT}}}\models E_3$,
Carlson uses transfinite recursion up to $\epsilon_0\cdot \omega$, as well
as deep structural properties (from \cite{carlson1999}) about the ordinals.
That $\mathscr N_{T_{\text{SMT}}}$ satisfies lines 2, 5, 6, and 7, is trivial; that it satisfies
line 4 follows from the fact that it satisfies lines 1--2.
Line 1 would be trivial if we added the following line to $T_{\text{SMT}}$:
\begin{itemize}
\item[1b.] (Assigned Validity) $\phi^s$, whenever $\phi$ is valid and $s$ is any assignment.
\end{itemize}
Theorems from \cite{carlson2000} imply Assigned Validity is already a consequence of $T_{\text{SMT}}$,
so this addition is not necessary, however it becomes necessary if (say) line 2 is weakened.

The main result in this paper is that by weakening $E_2$, removing $E_4$, and adding Assigned Validity, we remove the
need to induct up to $\epsilon_0\cdot\omega$.  Induction up to $\omega\cdot\omega$ suffices,
and the computations from \cite{carlson1999} can also be avoided.  This is surprising because
we do not weaken $E_3$, the lone schema for which such sophisticated 
methods were used before.

\begin{definition}
\label{depthdefn}
For any $\LEA$-formula $\phi$, let $\depth(\phi)$
denote the depth to which $K$ operators are nested in $\phi$, more formally:
\begin{itemize}
\item If $\phi$ is an $\LPA$-formula then $\depth(\phi)=0$.
\item If $\phi\equiv K(\phi_0)$ then $\depth(\phi)=\depth(\phi_0)+1$.
\item If $\phi\equiv (\rho\rightarrow\sigma)$ then $\depth(\phi)=\max\{\depth(\rho),\depth(\sigma)\}$.
\item If $\phi\in \{(\neg\phi_0), (\forall x\phi_0)\}$ then $\depth(\phi)=\depth(\phi_0)$.
\end{itemize}
\end{definition}

Now let $T^w_{\text{SMT}}$ be the $\LEA$-theory containing the following schemas:
\begin{enumerate}
\item $E_1$ and $E_3$.
\item Assigned Validity: $\phi^s$ whenever $\phi$ is valid and $s$ is any assignment.
\item ($E'_2$)  $\ucl{K(\phi\rightarrow\psi)\rightarrow K\phi\rightarrow K\psi}$
provided $\depth(\phi)\leq\depth(\psi)$.
\item The axioms of Epistemic Arithmetic.
\item Mechanicalness.
\item $K\phi$ whenever $\phi$ is an instance of lines 1--5 or (recursively) 6.
\end{enumerate}
Our main result (obtained by inducting only up to $\epom$)
will imply $T^w_{\text{SMT}}$ is true.
%
%

%
%

\section{Stratifiers}
\label{stratifierssectn}

\begin{definition}
Let $\Lom$ be the language obtained from $\LPA$ by adding operators $K^\alpha$ for all $\alpha\in\epom$.
For any $\Lom$-formula $\phi$, let
\[\onset(\phi) = \{\alpha\in\epom\,:\,\mbox{$K^\alpha$ occurs in $\phi$}\}.\]
\end{definition}

An example of an $\Lom$-formula: $\forall x(K^{\omega}K^{\omega\cdot 7+2}K^{53}K^0(x=0)\rightarrow K^{\omega\cdot 7+3}(x=0))$.

\begin{definition}
\label{stratifierdefn}
(Stratifiers)
For any infinite subset $X\subseteq\epom$, the \emph{stratifier given by $X$}
is the function $\bullet^+$ that takes $\LEA$-formulas to $\Lom$-formulas in the following way.
\begin{enumerate}
\item If $\phi$ is atomic, $\phi^+\equiv\phi$.
\item If $\phi$ is $\phi_1\rightarrow\phi_2$, $\neg\phi_1$, or $\forall x\phi_1$, then
$\phi^+$ is $\phi^+_1\rightarrow\phi^+_2$, $\neg\phi^+_1$, or $\forall x\phi^+_1$, respectively.
\item If $\phi$ is $K\phi_0$, then $\phi^+\equiv K^\alpha\phi^+_0$ where
$\alpha$ is the smallest ordinal in $X\backslash\onset(\phi^+_0)$.
\end{enumerate}
By a \emph{stratifier}, we mean a stratifier given by some $X$.
By the \emph{veristratifier}, we mean the stratifier given by $X=\{\omega\cdot 1,\omega\cdot 2,\ldots\}$.
If $\bullet^+$ is a stratifier and $T$ is an $\LEA$-theory,
$T^+$ denotes $\{\phi^+\,:\,\phi\in T\}$.
\end{definition}

For example, if $\bullet^+$ is the veristratifier, then
\[
\left(K(1=0)\rightarrow KK(1=0)\right)^+ \,\equiv\, K^\omega(1=0)\rightarrow K^{\omega\cdot 2}K^\omega(1=0).
\]

\begin{lemma}
\label{assignmentplayswellwithstratifier}
Suppose $\phi$ is an $\LEA$-formula, $s$ is an assignment into $\N$,
and $\bullet^+$ is a stratifier.
If $\alpha,\beta\in\epom$ are such that $(K\phi)^+\equiv K^\alpha\phi^+$
and $(K\phi^s)^+\equiv K^\beta(\phi^s)^+$, then $\alpha=\beta$.
\end{lemma}

\begin{proof}
By induction.
\end{proof}

\begin{lemma}
\label{depthandstratifier}
Suppose $\phi$ and $\psi$ are $\LEA$-formulas and $\bullet^+$ is a stratifier.
Let $\alpha,\beta\in\epom$ be such that $(K\phi)^+\equiv K^\alpha\phi^+$
and $(K\psi)^+\equiv K^\beta\psi^+$.
Then $\depth(\phi)<\depth(\psi)$ if and only if $\alpha<\beta$.
\end{lemma}

\begin{proof}
By induction.
\end{proof}

\begin{definition}
For any $\Lom$-structure $\M$ and stratifier $\bullet^+$, let $\M^+$ be the $\LEA$-structure
that has the same universe and interpretation of $\LPA$ as $\M$, and that interprets $K$ so that
for any $\LEA$-formula $\phi$ and assignment $s$,
\[
\mbox{$\M^+\models K\phi[s]$ if and only if $\M\models (K\phi)^+[s]$.}
\]
\end{definition}

It is easy to check that if $\M$ is an $\Lom$-structure then $\M^+$ really is an $\LEA$-structure (it satisfies
Conditions 1--3 of Definition \ref{baselogicdefn}).
From now on we will suppress this remark when defining new structures.

\begin{lemma}
\label{structuregrowingmagic}
Let $\M$ be an $\Lom$-structure, $\bullet^+$ a stratifier.   For any $\LEA$-formula $\phi$ and assignment $s$,
\[
\mbox{$\M^+\models\phi[s]$ if and only if $\M\models \phi^+[s].$}
\]
\end{lemma}

\begin{proof}
A straightforward induction.
\end{proof}

\begin{definition}
For any $\Lom$-formula $\phi$, $\phi^-$ is the $\LEA$-formula obtained by changing every operator of the form $K^\alpha$ in $\phi$
into $K$.
If $T$ is an $\Lom$-theory, $T^-=\{\phi^-\,:\,\phi\in T\}$.
\end{definition}

\begin{example}
$\left(K^{\omega\cdot 8+3}\forall x K^{17}(x=y)\right)^- \,\equiv\, K\forall x K(x=y).$
\end{example}

\begin{lemma}
\label{obviouslemma}
Let $\bullet^+$ be a stratifier.
For any $\LEA$-formula $\phi$, $(\phi^+)^-\equiv\phi$.
\end{lemma}

\begin{proof}
Straightforward.
\end{proof}

%
%

\begin{definition}
If $\M$ is an $\LEA$-structure, let $\M^-$ be the $\Lom$-structure that has the same universe as $\M$,
agrees with $\M$ on $\LPA$, and interprets each $K^\alpha$
so that for any $\Lom$-formula $\phi$ and assignment $s$,
\[
\mbox{$\M^-\models K^\alpha\phi[s]$ if and only if $\M\models K\phi^-[s]$.}
\]
\end{definition}

In \cite{carlson2000} (Definition 5.4), $\M^-$ is the \emph{stratification of $\M$ over $\epom$}.

\begin{lemma}
\label{structureshrinkingmagic}
For any $\LEA$-structure $\M$, $\Lom$-formula $\phi$ and assignment $s$,
\[
\mbox{$\M^-\models\phi[s]$ if and only if $\M\models\phi^-[s]$.}
\]
\end{lemma}

\begin{proof}
A straightforward induction.
\end{proof}

\begin{theorem}
\label{stratifiersrespectvalidity}
\item
\begin{enumerate}
\item For any valid $\Lom$-formula $\phi$, $\phi^-$ is valid.
\item
For any $\LEA$-formula $\phi$ and stratifier $\bullet^+$, $\phi$ is valid if and only if $\phi^+$ is valid.
\end{enumerate}
\end{theorem}

\begin{proof}
\item
(1) Let $\phi$ be a valid $\Lom$-formula.
For any $\LEA$-structure $\M$ and assignment $s$, since $\phi$ is valid, $\M^-\models\phi[s]$
and so by Lemma \ref{structureshrinkingmagic}, $\M\models\phi^-[s]$.
By arbitrariness of $\M$ and $s$, $\phi^-$ is valid.

\item
(2, $\Rightarrow$)
Assume $\phi$ is a valid $\LEA$-formula.
For any $\Lom$-structure $\M$ and assignment $s$, since $\phi$ is valid, $\M^+\models\phi[s]$,
and $\M\models\phi^+[s]$ by Lemma \ref{structuregrowingmagic}.  By arbitrariness of $\M$ and $s$,
this shows $\phi^+$ is valid.

\item
(2, $\Leftarrow$)
Assume $\phi$ is an $\LEA$-formula and $\phi^+$ is valid.
For any $\LEA$-structure $\M$ and assignment $s$, since $\phi^+$ is valid, $\M^-\models\phi^+[s]$,
and $\M\models (\phi^+)^-[s]$ by Lemma \ref{structureshrinkingmagic}.
By Lemma \ref{obviouslemma}, $\M\models\phi[s]$.  By arbitrariness of $\M$ and $s$,
$\phi$ is valid.
\end{proof}

\begin{definition}
\label{oplusdefn}
For any $\LEA$-theory $T$, let
\[
T^\oplus=\{\phi^+\,:\,\mbox{$\phi\in T$ and $\bullet^+$ is a stratifier}\}.
\]
\end{definition}

\begin{example}
Suppose $T$ is the $\LEA$-theory consisting of $K\phi\rightarrow KK\phi$ for all $\LPA$-sentences $\phi$.
Then $T^\oplus$ is the $\Lom$-theory consisting of
$K^\alpha\phi\rightarrow K^\beta K^\alpha\phi$ for all $\LPA$-sentences $\phi$ and ordinals $\alpha<\beta<\epom$.
\end{example}

\begin{theorem}
\label{proofstrat}
(Upward proof stratification)
For any $\LEA$-theory $T$, $\LEA$-sentence $\phi$, and stratifier $\bullet^+$,
the following are equivalent.
\begin{align*}
\mbox{1. $T\models\phi$.} && \mbox{2. $T^+\models\phi^+$.} && \mbox{3. $T^\oplus\models\phi^+$.}
\end{align*}
%
%
%
%
\end{theorem}

This theorem is so-named because it is an upside-down version of a harder theorem
that we called \cite{alexanderdissert} \emph{proof stratification}.
%
%
In non-upward proof stratification, $T$ and $\phi$ are taken in the \emph{stratified} language
and the theorem states that $T\models\phi$ if and only if $T^-\models\phi^-$.
This uses complicated hypotheses on $T$ and $\phi$.
Versions of these hypotheses could be stated in an elementary way,
but \emph{a priori} they might imply $T$ is inconsistent (in which case Theorem \ref{proofstrat}
is trivial).  The only way we know to exhibit consistent theories that satisfy such hypotheses
is to exploit the machinery from \cite{carlson1999} on the $\Sigma_1$-structure of the ordinals.

\begin{proof}[Proof of Theorem~\ref{proofstrat}]
Let $T$, $\phi$ and $\bullet^+$ be as in Theorem \ref{proofstrat}.

\item
($1\Rightarrow 2$) Assume $T\models\phi$.  By Theorem \ref{completenesscompactness},
there are $\tau_1,\ldots,\tau_n\in T$ such that $\left(\bigwedge_i\tau_i\right)\rightarrow\phi$
is valid.  By Theorem \ref{stratifiersrespectvalidity},
$\left(\bigwedge_i\tau^+_i\right)\rightarrow\phi^+$ is valid,
showing $T^+\models\phi^+$.

\item
($2\Rightarrow 3$) Trivial: $T^+\subseteq T^\oplus$.

\item
($3\Rightarrow 1$)
Assume $T^\oplus\models\phi^+$.
By Theorem \ref{completenesscompactness}
there are $\tau_1,\ldots,\tau_n\in T^\oplus$
such that $\left(\bigwedge_i\tau_i\right)\rightarrow\phi^+$ is valid.
By definition of $T^\oplus$ there are $\sigma_1,\ldots,\sigma_n\in T$
and stratifiers $\bullet^1,\ldots,\bullet^n$
such that each $\tau_i\equiv\sigma^i_i$.
By Lemma \ref{obviouslemma}
\[
\mbox{$\left(\left(\bigwedge_i\sigma^i_i\right)\rightarrow\phi^+\right)^- \,\equiv\,
      \left(\bigwedge_i\sigma_i  \right)\rightarrow\phi$},
\]
so Theorem \ref{stratifiersrespectvalidity} guarantees
$\left(\bigwedge_i\sigma_i\right)\rightarrow\phi$
is valid, and $T\models\phi$.
\end{proof}

\section{Uniform Theories and Collapsing Knowledge}
\label{uniformsect}

\begin{definition}
Suppose $X\subseteq \epom$ and $h:X\to\epom$.
For any $\Lom$-formula $\phi$, we define $h(\phi)$ to be the $\Lom$-formula
obtained by replacing $K^\alpha$ by $K^{h(\alpha)}$ everywhere $K^{\alpha}$
occurs in $\phi$ ($\alpha\in X$).
(If $\alpha\not\in X$, we do not change occurrences of $K^\alpha$ in $\phi$.)
\end{definition}

\begin{example}Suppose $\alpha_1<\cdots<\alpha_4$ are distinct ordinals in $\omega\cdot\omega$.
Let $X=\{\alpha_2,\alpha_3\}$, let $h(\alpha_2)=\alpha_3$, $h(\alpha_3)=\alpha_4$.
Then
\[
h\left(K^{\alpha_3}K^{\alpha_2}K^{\alpha_1}(1=1)\right) \,\equiv\,
K^{\alpha_4}K^{\alpha_3}K^{\alpha_1}(1=1).
\]
\end{example}


\begin{definition}
\label{uniformdefn}
An $\Lom$-theory $T$ is \emph{uniform}
if the following statement holds.
For all $X\subseteq\epom$, for all order-preserving $h:X\to\epom$,
for all $\phi\in T$, if $\onset(\phi)\subseteq X$ then $h(\phi)\in T$.
\end{definition}

\begin{example}If $T$ contains $K^1K^0(1=0)$ and $T$ is uniform, then $T$ must contain
$K^\beta K^\alpha(1=0)$ for all $\alpha<\beta\in\epom$.
\end{example}

\begin{lemma}
\label{uniformityofstratifiers}
Suppose $\bullet^+$ is a stratifier, $X\subseteq\epom$,
$h:X\to\epom$ is order preserving,
and $\phi$ is an $\LEA$-formula with $\onset(\phi^+)\subseteq X$.
There is a stratifier $\bullet^*$ such that $\phi^*\equiv h(\phi^+)$.
\end{lemma}

\begin{proof}
Let $Y_0=\{h(\alpha)\,:\,\alpha\in \onset(\phi^+)\}$, $Y=Y_0\cup\{\beta\in\epom\,:\,\beta>Y_0\}$,
and let $\bullet^*$ be the stratifier given by $Y$.
By induction, for every subformula $\phi_0$ of $\phi$, $\phi_0^*\equiv h(\phi_0^+)$.
\end{proof}

\begin{lemma}
\label{uniformitylemma}
(Uniformity lemma)
For any $\LEA$-theory $T$, $T^\oplus$ is uniform.
\end{lemma}

\begin{proof}
Let $X\subseteq\epom$, let $h:X\to\epom$ be order preserving, let $\phi\in T^\oplus$,
and assume $\onset(\phi)\subseteq X$.
By definition of $T^\oplus$, $\phi\equiv\phi^+_0$ for some $\phi_0\in T$ and some stratifier $\bullet^+$.
By Lemma \ref{uniformityofstratifiers} there is a stratifier $\bullet^*$ such that $h(\phi^+_0)\equiv \phi^*_0$.
This shows $h(\phi)\in T^\oplus$.
\end{proof}

Unfortunately, the range of $\oplus$ does not include every uniform $\Lom$-theory.
For example, suppose $T$ is the $\Lom$-theory consisting of
\[K^\alpha(\phi^+\rightarrow\psi^+)\rightarrow K^\alpha\phi^+\rightarrow K^\alpha\psi^+\]
for all $\LEA$-sentences $\phi$ and $\psi$ and stratifiers $\bullet^+$
with $\onset(\phi^+),\onset(\psi^+)<\alpha\in\epom$.
The reader may check that despite being uniform, $T$ is not $T^\oplus_0$ for any $\LEA$-theory $T_0$. 

\begin{definition}
\label{structuremappingdefn}
If $\M$ is an $\Lom$-structure, $X\subseteq \epom$, and $h:X\to\epom$,
we define an $\Lom$-structure $h(\M)$ that has the same universe as $\M$,
agrees with $\M$ on the interpretation of $\LPA$, and interprets
$K^\alpha$ so that for any $\Lom$-formula $\phi$ and assignment $s$,
\[
\mbox{$h(\M)\models K^\alpha\phi[s]$ if and only if $\M\models h(K^\alpha\phi)[s]$.}
\]
\end{definition}

\begin{lemma}
\label{structuremappingmagic}
Suppose $\M$, $X$, and $h$ are as in Definition \ref{structuremappingdefn}.
For any $\Lom$-formula $\phi$ and assignment $s$,
\[
\mbox{$h(\M)\models\phi[s]$
if and only if $\M\models h(\phi)[s]$.}
\]
\end{lemma}

\begin{proof}
By induction.
\end{proof}

We will only need part 1 of the next lemma, we state part 2 for completeness.

\begin{lemma}
\label{hpreservesvalidity}
Suppose $\M$, $X$, and $h$ are as in Definition \ref{structuremappingdefn} and $\phi$ is an $\Lom$-formula.
\begin{enumerate}
\item
If $\phi$ is valid then $h(\phi)$ is valid.
\item
Assume $h$ is injective.
If $\onset(\phi)\subseteq X$ and $h(\phi)$ is valid, then $\phi$ is valid.
\end{enumerate}
\end{lemma}

\begin{proof}
\item
(1) Similar to Theorem \ref{stratifiersrespectvalidity}.

\item
(2) If $h(\phi)$ is valid then $h^{-1}(h(\phi))$ is valid by part 1.  Since $\onset(\phi)\subseteq X$, $h^{-1}(h(\phi))\equiv\phi$.
\end{proof}

\begin{definition}
For any $\Lom$-theory $T$ and $\alpha\in\epom$,
let $T\cap\alpha=\{\phi\in T\,:\,\onset(\phi)\subseteq\alpha\}$
be the subset of $T$ where all superscripts are strictly bounded by $\alpha$.
\end{definition}

\begin{example}
\item
\begin{itemize}
\item
For any $\Lom$-theory $T$, $T\cap 0=\{\phi\in T\,:\,\mbox{$\phi$ is an $\LPA$-sentence}\}$.
\item
For any $\Lom$-theory $T$, $T\cap 1=\{\phi\in T\,:\,\mbox{$\phi$ is an 
$\LPA\cup\{K^0\}$-sentence}\}$.
\item
For any $\LEA$-theory $T$,
$T^\oplus\cap\omega = \{\phi^+\,:\,\mbox{$\phi\in T$ and $\bullet^+$ is given by some $X\subseteq\omega$}\}$.
\end{itemize}
\end{example}


\begin{theorem}
\label{collapsethm}
(The collapse theorem)
Let $T$ be a uniform $\Lom$-theory.
For any $0<n\in\N$ and $\Lom$-formula $\phi$ with $\onset(\phi)\subseteq\omega\cdot n$,
$T\models\phi$ if and only if $T\cap(\omega\cdot n)\models\phi$.
\end{theorem}
%
%
%
%

\begin{proof}
\setlength{\columnsep}{-2.25in}
\begin{multicols}{2}
The $\Leftarrow$ direction is trivial: $T\cap(\omega\cdot n)\subseteq T$.
For $\Rightarrow$, assume $T\models \phi$.
By Theorem \ref{completenesscompactness} there are $\tau_1,\ldots,\tau_n\in T$
such that
\[
\mbox{$\Phi\equiv \left(\bigwedge_i \tau_i\right)\rightarrow\phi$}
\]
is valid.
Let $X=\onset(\Phi)\cap(\omega\cdot n)$, $Y=\onset(\Phi)\cap[\omega\cdot n,\infty)$,
see Fig.~1.
Then $|X|,|Y|<\infty$ and $X\cup Y=\onset(\Phi)$.

\columnbreak

\hspace{2.15in}
{
\vspace{2mm}
\includegraphics[scale=.65]{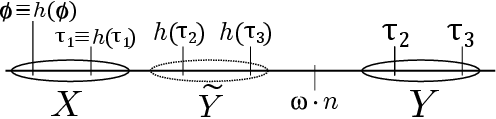}

\vspace{-2.15mm}\hspace{2.65in}{
\parbox[t][1in]{1.5in}{
Figure 1:  Collapse.
}
}
\par
}

\end{multicols}

\vspace{-.685in}
Since $|X|<\infty$ and $\omega\cdot n$ has no maximum element, there are infinitely many
ordinals above $X$ in $\omega\cdot n$.
Thus since $|Y|<\infty$ we can find $\widetilde{Y}\subseteq\omega\cdot n$ such that
$X<\widetilde{Y}$ and $|\widetilde{Y}|=|Y|$.
It follows there is an order preserving function $h:X\cup Y\to X\cup \widetilde{Y}$
such that $h(x)=x$ for all $x\in X$.

By Lemma \ref{hpreservesvalidity}, $h(\Phi)$ is valid.
Since $\onset(\phi)\subseteq \omega\cdot n$, we have $\onset(\phi)\subseteq X$ and $h(\phi)\equiv\phi$.
Thus
\[
\mbox{$h(\Phi) \,\equiv\,\left(\bigwedge_i h(\tau_i)\right)\rightarrow h(\phi) \,\equiv\,\left(\bigwedge_i h(\tau_i)\right)\rightarrow\phi$}.
\]
Since $T$ is uniform, each $h(\tau_i)\in T$.
In fact, since $\mathrm{range}(h)\subseteq\omega\cdot n$, each $h(\tau_i)\in T\cap(\omega\cdot n)$,
and the validity of $\left(\bigwedge_i h(\tau_i)\right)\rightarrow\phi$ witnesses $T\cap(\omega\cdot n)\models\phi$.
\end{proof}

%
%
%
%
%

\begin{definition}
\label{stratifiedmodel}
If $T$ is an $\Lom$-theory, its intended structure is the $\Lom$-structure
$\M_T$ with standard first-order part that interprets the operators of $\Lom$
so that for every $\Lom$-formula $\phi$, assignment $s$, and $\alpha\in\epom$,
\[
\mbox{$\M_T \models K^\alpha\phi[s]$ if and only if $T\cap\alpha\models \phi^s$.}
\]
\end{definition}

\begin{lemma}
\label{scriptmbehavesasintended}
Suppose $T$ is an $\Lom$-theory.
For any $\Lom$-formula $\phi$ and assignment $s$,
$\M_T\models\phi[s]$ if and only if $\M_T\models\phi^s$.
\end{lemma}

\begin{proof}
By induction.
\end{proof}

Recall from Definition \ref{stratifierdefn} that the veristratifier
is the stratifier given by $X=\{\omega\cdot1,\omega\cdot2,\ldots\}$.

\begin{theorem}
\label{upwardstratificationtheorem}
(The upward stratification theorem)
Let $\bullet^+$ be the veristratifier.
For any $\LEA$-theory $T$, $\LEA$-formula $\phi$, and assignment $s$,
$\mathscr N_T\models\phi[s]$ if and only if $\M_{T^\oplus}\models\phi^+[s]$.
\end{theorem}

Again, the theorem is so-named because it is an upside-down version of
a harder theorem that equates $\M_T\models\phi[s]$ with $\mathscr N_{T^-}\models\phi^-[s]$
for stratified $T$ and $\phi$ under more complicated 
hypotheses.

\begin{proof}[Proof of Theorem \ref{upwardstratificationtheorem}]
By induction on $\phi$.
The only nontrivial case is when $\phi$ is $K\psi$.
Then $\phi^+\equiv K^\alpha\psi^+$ for some $\alpha$.
By definition of the veristratifier, $\alpha=\omega\cdot n$
for some $0<n\in\N$, and $\onset(\psi^+)\subseteq\omega\cdot n$.
By Lemma \ref{uniformitylemma}, $T^\oplus$ is uniform,
so we can use the collapse theorem (Theorem \ref{collapsethm}).
The following are equivalent.
\begin{align*}
\mathscr N_T &\models K\psi[s]\\
T &\models \psi^s
  &\mbox{(Definition \ref{defnofintendedmodel})}\\
T^\oplus &\models (\psi^s)^+
  &\mbox{(Upward proof stratification---Theorem \ref{proofstrat})}\\
T^\oplus\cap(\omega\cdot n) &\models (\psi^s)^+
  &\mbox{(The collapse theorem---Theorem \ref{collapsethm})}\\
T^\oplus\cap(\omega\cdot n) &\models (\psi^+)^s
  &\mbox{(Clearly $(\psi^s)^+\equiv(\psi^+)^s$)}\\
\M_{T^\oplus} &\models K^{\omega\cdot n}\psi^+[s].
  &\mbox{(Definition \ref{stratifiedmodel})} 
\end{align*}
\end{proof}

\begin{corollary}
\label{revelatorycorollary}
For any $\LEA$-theory $T$,
in order to show $\mathscr N_T\models T$, it suffices
to show $\M_{T^\oplus}\models T^\oplus$.
\end{corollary}

Corollary \ref{revelatorycorollary} provides a foothold 
for proving truth of self-referential theories
by transfinite induction up to $\omega\cdot\omega$:
in order to prove $\mathscr N_{T}\models T$,
one can attempt to prove $\M_{T^\oplus}\models T^\oplus\cap\alpha$
for all $\alpha\in\epom$ by induction on $\alpha$.

\section{Upward Generic Axioms}
\label{genericaxiomssectn}

One way to state an epistemological consistency result, for example
that a truthful machine can know itself to be true and recursively enumerable, is to show that
the schemas in question are consistent with a particular background theory of knowledge.
We take a more general approach:
show that the doubted schemas are consistent with \emph{any} background theory satisfying certain
conditions.
%
%

We say that an $\LEA$-theory $T$ is \emph{$K$-closed} if $K\phi\in T$ whenever $\phi\in T$.

\begin{definition}
\label{baregenericdefn}
Suppose $T_0$ is an $\LEA$-theory.
\begin{enumerate}
\item $T_0$ is \emph{generic} if $\mathscr N_T\models T_0$ for every $\LEA$-theory $T\supseteq T_0$.
\item $T_0$ is \emph{closed-generic} if $T_0$ is $K$-closed and $\mathscr N_T\models T_0$ for every $K$-closed $\LEA$-theory $T\supseteq T_0$.
\item $T_0$ is \emph{r.e.-generic} if $T_0$ is r.e.~and $\mathscr N_T\models T_0$ for every r.e.~$\LEA$-theory $T\supseteq T_0$.
\item $T_0$ is \emph{closed-r.e.-generic} if $T_0$ is $K$-closed, r.e., and $\mathscr N_T\models T_0$ for every $K$-closed r.e.~$\LEA$-theory
$T\supseteq T_0$.
\end{enumerate}
\end{definition}

\begin{lemma}
\label{genericdiamond}
\item
\begin{enumerate}
\item Generic$+$r.e.~implies r.e.-generic.
\item Generic$+K$-closed implies closed-generic.
\item Closed-generic$+$r.e.~implies closed-r.e.-generic.
\item R.e.-generic$+K$-closed implies closed-r.e.-generic.
\end{enumerate}
\end{lemma}

\begin{proof}
Straightforward.
\end{proof}

\begin{lemma}
\label{genericunions}
Let $T=\cup_{i\in I} T_i$ where each $T_i$ is an $\LEA$-theory.
\begin{enumerate}
\item If the $T_i$ are generic, then $T$ is generic.
\item If the $T_i$ are closed-generic, then $T$ is closed-generic.
\item If the $T_i$ are r.e.-generic and $T$ is r.e., then $T$ is r.e.-generic.
\item If the $T_i$ are closed-r.e.-generic and $T$ is r.e., then $T$ is closed-r.e.-generic.
\end{enumerate}
\end{lemma}

\begin{proof}
Straightforward.
\end{proof}

\begin{lemma}
\label{etwoisgeneric}
The $\LEA$-schema $E_2$, consisting of $\ucl{K(\phi\rightarrow\psi)\rightarrow K\phi\rightarrow K\psi}$, is generic.
\end{lemma}

\begin{proof}
Suppose $T\supseteq E_2$ is arbitrary.
For any $\LEA$-formulas $\phi$ and $\psi$ and assignment $s$, if
$\mathscr N_T\models K(\phi\rightarrow\psi)[s]$ and $\mathscr N_T\models K\phi[s]$,
then
\begin{align*}
T &\models \phi^s\rightarrow\psi^s
  &\mbox{(Definition \ref{defnofintendedmodel})}\\
T &\models \phi^s
  &\mbox{(Definition \ref{defnofintendedmodel})}\\
T &\models \psi^s
  &\mbox{(Modus Ponens)}\\
\mathscr N_T &\models K\psi[s],\mbox{ as desired.}
  &\mbox{(Definition \ref{defnofintendedmodel})} 
\end{align*}
\end{proof}

\begin{definition}
Suppose $T_0$ is an $\LEA$-theory.
\begin{enumerate}
\item $T_0$ is \emph{upgeneric} if $\M_{T^\oplus}\models T^\oplus_0$ for every $\LEA$-theory $T\supseteq T_0$.
\item $T_0$ is \emph{closed-upgeneric} if $T_0$ is $K$-closed and $\M_{T^\oplus}\models T^\oplus_0$ for every $K$-closed $\LEA$-theory $T\supseteq
T_0$.
\item $T_0$ is \emph{r.e.-upgeneric} if $T_0$ is r.e.~and $\M_{T^\oplus}\models T^\oplus_0$ for every r.e.~$\LEA$-theory $T\supseteq T_0$.
\item $T_0$ is \emph{closed-r.e.-upgeneric} if $T_0$ is
$K$-closed, r.e., and $\M_{T^\oplus}\models T^\oplus_0$ for every $K$-closed r.e.~$\LEA$-theory $T\supseteq T_0$.
\end{enumerate}
\end{definition}

\begin{lemma}
(Compare Lemma \ref{genericdiamond})
\begin{enumerate}
\item Upgeneric$+K$-closed implies closed-generic.
\item Upgeneric$+$r.e.~implies r.e.-upgeneric.
\item Closed-upgeneric$+$r.e.~implies closed-r.e.-upgeneric.
\item R.e.-upgeneric$+K$-closed implies closed-r.e.-upgeneric.
\end{enumerate}
\end{lemma}

\begin{proof}
Straightforward.
\end{proof}

\begin{lemma}
\label{upgenericunions}
Suppose $T=\cup_{i\in I}T_i$ where the $T_i$ are $\LEA$-theories.
\begin{enumerate}
\item If the $T_i$ are upgeneric, then $T$ is upgeneric.
\item If the $T_i$ are closed-upgeneric, then $T$ is closed-upgeneric.
\item If the $T_i$ are r.e.-upgeneric and $T$ is r.e., then $T$ is 
r.e.-upgeneric.
\item If the $T_i$ are closed-r.e.-upgeneric and $T$ is r.e., then $T$ is 
closed-r.e.-upgeneric.
\end{enumerate}
\end{lemma}

\begin{proof}
Straightforward.
\end{proof}

\begin{lemma}
\label{upgenericimpliesgeneric}
\item
\begin{enumerate}
\item Upgeneric implies generic.
\item Closed-upgeneric implies closed-generic.
\item R.e.-upgeneric implies r.e.-generic.
\item Closed-r.e.-upgeneric implies closed-r.e.-generic.
\end{enumerate}
\end{lemma}

\begin{proof}
By the upward stratification theorem (Theorem \ref{upwardstratificationtheorem}).
\end{proof}

In light of Lemmas \ref{etwoisgeneric} and \ref{upgenericimpliesgeneric},
the following shows that upgeneric is strictly
stronger than generic.

\begin{lemma}
\label{etwonotupgeneric}
$E_2$ is not upgeneric.  In fact $E_2$ is not even closed-r.e.-upgeneric.
\end{lemma}

\begin{proof}
Let $T$ be the smallest $K$-closed $\LEA$-theory containing the following schemata.
\begin{enumerate}
\item $E_2$.
\item $K(1=0)$.
\item $K(1=0)\rightarrow (1=0)$.
\end{enumerate}
Since $T\supseteq E_2$ is closed r.e., it suffices to exhibit
some $\theta\in E_2$ and stratifier $\bullet^+$ such that $\M_{T^\oplus}\not\models \theta^+$.
If $\bullet^+$ is the stratifier given by $X=\{0,1,2,\ldots\}$,
%
%
the reader can check that
\[
\theta \,\equiv\,\,\, K(K(1=0)\rightarrow (1=0))\rightarrow KK(1=0)\rightarrow K(1=0)
\]
works.
\end{proof}

Lemma \ref{etwonotupgeneric} and
the following demystify our reason for weakening $E_2$ to $E'_2$.

\begin{lemma}
\label{etwoprimeisupgeneric}
The schema $E'_2$, consisting of $\ucl{K(\phi\rightarrow\psi)\rightarrow K\phi\rightarrow K\psi}$
whenever $\depth(\phi)\leq\depth(\psi)$ (Definition \ref{depthdefn}), is upgeneric.
\end{lemma}

\begin{proof}
Let $T\supseteq E'_2$ be arbitrary.
Suppose $\phi$ and $\psi$ are $\LEA$-formulas with $\depth(\phi)\leq\depth(\psi)$
and $\bullet^+$ is a stratifier,
say with
\begin{align*}
(K\phi)^+ &\equiv K^\alpha\phi^+\\
(K\psi)^+ &\equiv K^\beta\psi^+\\
(K(\phi\rightarrow \psi))^+ &\equiv K^\gamma(\phi^+\rightarrow \psi^+),
\end{align*}
we will show $\M_{T^\oplus}$ satisfies
\[
(\ucl{K(\phi\rightarrow\psi)\rightarrow K\phi\rightarrow K\psi})^+
\equiv
\ucl{K^\gamma(\phi^+\rightarrow\psi^+)\rightarrow 
K^\alpha\phi^+\rightarrow 
K^\beta\psi^+}.
\]
Note that by Lemma \ref{depthandstratifier}, $\alpha\leq\beta=\gamma$.
Let $s$ be an arbitrary assignment such that
$\M_{T^\oplus}\models K^\gamma(\phi^+\rightarrow\psi^+)[s]$ and $\M_{T^\oplus}\models K^\alpha\phi^+[s]$.
Then
\begin{align*}
T^\oplus\cap\gamma &\models (\phi^+)^s \rightarrow (\psi^+)^s
  &\mbox{(Definition \ref{stratifiedmodel})}\\
T^\oplus\cap\alpha &\models (\phi^+)^s
  &\mbox{(Definition \ref{stratifiedmodel})}\\
T^\oplus\cap\beta &\models ((\phi^+)^s\rightarrow (\psi^+)^s)\,\wedge\, (\phi^+)^s
  &\mbox{(Since $\alpha\leq\beta=\gamma$)}\\
T^\oplus\cap\beta &\models (\psi^+)^s
  &\mbox{(Modus Ponens)}\\
\M_{T^\oplus} &\models K^\beta \psi^+[s],\mbox{ as desired.}
  &\mbox{(Definition \ref{stratifiedmodel})} 
\end{align*}
\end{proof}

\begin{lemma}
\label{assignedvalidityisupgeneric}
The Assigned Validity schema, consisting of $\phi^s$ whenever $\phi$ is valid and $s$ is any assignment,
is upgeneric.
\end{lemma}

\begin{proof}
Let $T\supseteq\mbox{(Assigned Validity)}$ be arbitrary.
Suppose $\phi$ is valid, $s$ is an assignment, and $\bullet^+$ is a stratifier.
By Theorem \ref{stratifiersrespectvalidity}, $\phi^+$ is also valid.
Thus $\M_{T^\oplus}\models\phi^+[s]$, and by Lemma \ref{scriptmbehavesasintended},
$\M_{T^\oplus}\models (\phi^+)^s$.
By arbitrariness of $\phi$, $s$, and $\bullet^+$, $\M_{T^\oplus}\models \mbox{(Assigned Validity)}^\oplus$.
\end{proof}

\begin{lemma}
\label{trivialgenericlemma}
Any set of true purely arithmetical sentences is upgeneric.
\end{lemma}

\begin{proof}
Trivial: $\M_T$ has standard first-order part.
\end{proof}

\begin{lemma}
\label{eaisupgeneric}
The schema consisting of the axioms of Epistemic Arithmetic (Peano Arithmetic with induction extended to 
$\LEA$) is upgeneric.
\end{lemma}

\begin{proof}
Let $T\supseteq\mbox{(Epistemic Arithmetic)}$.
Let $\sigma$ be an axiom of Epistemic Arithmetic, $\bullet^+$ a stratifier.
If $\sigma$ is not an induction instance, then
$\M_{T^\oplus}\models \sigma^+$ by Lemma \ref{trivialgenericlemma}.
But suppose $\sigma$ is an instance
\[
\ucl{\phi(x|0)\rightarrow\forall x(\phi\rightarrow\phi(x|S(x)))\rightarrow\forall x\phi}
\]
of induction, so that $\sigma^+$ is
$\ucl{\phi^+(x|0)\rightarrow \forall x(\phi^+\rightarrow\phi^+(x|S(x)))\rightarrow\forall x\phi^+}$.
To show $\M_{T^\oplus}\models\sigma^+$, let $s$ be an assignment and assume
$\M_{T^\oplus}\models \phi^+(x|0)[s]$ and
$\M_{T^\oplus}\models\forall x(\phi^+\rightarrow\phi^+(x|S(x)))[s]$.
Then
\begin{align*}
\M_{T^\oplus} &\models \phi^+(x|0)^s
  &\mbox{(Lemma \ref{scriptmbehavesasintended})}\\
\M_{T^\oplus} &\models (\phi^+)^{s(x|0)}
  &\mbox{(Clearly $\psi(x|0)^s\equiv \psi^{s(x|0)}$)}\\
\forall n\in\N,\mbox{ if }\M_{T^\oplus}\models \phi^+[s(x|n)], &\mbox{ then }
\M_{T^\oplus}\models \phi^+(x|S(x))[s(x|n)]
  &\mbox{(First-order semantics of $\forall$ and $\rightarrow$)}\\
\forall n\in\N,\mbox{ if }\M_{T^\oplus}\models (\phi^+)^{s(x|n)}, &\mbox{ then }
\M_{T^\oplus}\models (\phi^+(x|S(x)))^{s(x|n)}
  &\mbox{(Lemma \ref{scriptmbehavesasintended})}\\
\forall n\in\N,\mbox{ if }\M_{T^\oplus}\models (\phi^+)^{s(x|n)}, &\mbox{ then }
\M_{T^\oplus}\models (\phi^+)^{s(x|n+1)}
  &\mbox{(Clearly $\psi(x|S(x))^{s(x|n)}\equiv \psi^{s(x|n+1)}$)}\\
\forall n\in\N,\mbox{ }\M_{T^\oplus} &\models (\phi^+)^{s(x|n)}
  &\mbox{(Mathematical induction)}\\
\forall n\in\N,\mbox{ }\M_{T^\oplus} &\models (\phi^+)[s(x|n)]
  &\mbox{(Lemma \ref{scriptmbehavesasintended})}\\
\M_{T^\oplus} &\models \forall x\phi^+[s]\mbox{, as desired.}
  &\mbox{(First-order semantics of $\forall$)} 
\end{align*}
\end{proof}

Armed with Lemmas \ref{genericunions} and \ref{upgenericunions},
computations such as Lemmas \ref{etwoisgeneric}, \ref{etwoprimeisupgeneric},
\ref{assignedvalidityisupgeneric},
\ref{trivialgenericlemma} and \ref{eaisupgeneric}
can be used as building blocks for background theories of knowledge.
Often, schemas we would like as building blocks are not (up)generic in isolation,
but become so when paired with other building blocks, as in the following three
lemmas.

\begin{lemma}
\label{eoneandassignedvalidityisupgeneric}
$E_1\cup (\mbox{Assigned Validity})$ is upgeneric
($E_1$ consists of $\ucl{K\phi}$ whenever $\phi$ is valid).
\end{lemma}

\begin{proof}
Let $T\supseteq E_1\cup (\mbox{Assigned Validity})$.
By Lemma \ref{assignedvalidityisupgeneric},
$\M_{T^\oplus}\models\mbox{(Assigned Validity)}^\oplus$,
we need only show $\M_{T^\oplus}\models E^\oplus_1$.
Let $\phi$ be valid, $\bullet^+$ any stratifier, and $s$ any assignment.
Since $T\supseteq (\mbox{Assigned Validity})$,
$T^\oplus$ contains the instance
\[(\phi^s)^+\equiv (\phi^+)^s\] of $(\mbox{Assigned Validity})^\oplus$.
In fact, $T^\oplus\cap\alpha$ contains $(\phi^+)^s$,
where $\alpha$ is such that $(K\phi)^+\equiv K^\alpha\phi^+$.
Thus by Definition \ref{stratifiedmodel}, $\M_{T^\oplus}\models K^{\alpha}\phi^+[s]$,
that is, $\M_{T^\oplus}\models (K\phi)^+[s]$.
This shows $\M_{T^\oplus}\models E^\oplus_1$.
\end{proof}

\begin{lemma}
\label{kclosurelemma}
For any upgeneric $T_0$, $T_0\cup K(T_0)$ is upgeneric,
where $K(T_0)$ consists of $K\phi$ whenever $\phi\in T_0$.
Similarly with ``upgeneric'' replaced by ``r.e.-upgeneric'',
``closed-upgeneric'', ``closed-r.e.-upgeneric'', ``generic'',
``r.e.-generic'', ``closed-generic'', or ``closed-r.e.-generic'' throughout.
\end{lemma}

\begin{proof}
We prove the upgeneric statement.
Suppose $T_0$ is upgeneric
and $T\supseteq T_0\cup K(T_0)$.
Since $T_0$ is upgeneric and $T\supseteq T_0$, $\M_{T^\oplus}\models T^\oplus_0$.
It remains to show $\M_{T^\oplus}\models (K\phi)^+$
for any sentence $\phi\in T_0$ and stratifier $\bullet^+$.
Let $\alpha$ be such that $(K\phi)^+\equiv K^\alpha\phi^+$.
By Definition \ref{stratifierdefn}, $\onset(\phi^+)\subseteq \alpha$
and thus $\phi^+\in T_0^\oplus\cap\alpha\subseteq T^\oplus\cap\alpha$.
Since $T^\oplus\cap\alpha\models \phi^+$, $\M_{T^\oplus}\models K^\alpha\phi^+$
as desired.
\end{proof}

We will not use the following lemma,
but it illuminates differences between this paper's
upward approach
and Carlson's original downward approach.

\begin{lemma}
\label{everythingworksgivenetwo}
$E_1\cup E_2\cup E_4\cup\mbox{(Epistemic Arithmetic)}$ is closed-generic.
\end{lemma}

\begin{proof}
Let $T$ be a $K$-closed theory containing $E_1$, $E_2$, $E_4$ and $\mbox{(Epistemic Arithmetic)}$.

By Lemma \ref{etwoisgeneric}, $\mathscr N_T\models E_2$.
By Lemmas \ref{eaisupgeneric} and \ref{upgenericimpliesgeneric}, $\mathscr N_T\models \mbox{(Epistemic Arithmetic)}$.
It remains to show $\mathscr N_T\models E_1\cup E_4$.  We will show $\mathscr N_T\models E_4$
and sketch $\mathscr N_T\models E_1$.

The typical sentence in $E_4$ is $\ucl{K\phi\rightarrow KK\phi}$.
Let $s$ be an assignment and assume $\mathscr N_T\models K\phi[s]$.
Then
\begin{align*}
T &\models \phi^s
  &\mbox{(Definition \ref{defnofintendedmodel})}\\
\exists \tau_1,\ldots,\tau_n\in T &\mbox{ s.t.}
\left(\wedge_{i=1}^n \tau_i\right)\rightarrow \phi^s
\mbox{ is valid}
  &\mbox{(Theorem \ref{completenesscompactness})}\\
T &\models K\left(\left(\wedge_{i=1}^n \tau_i\right)\rightarrow \phi^s\right)
  &\mbox{($T$ contains $E_1$)}\\
T &\models \left(\wedge_{i=1}^n K(\tau_i)\right)\rightarrow K\phi^s
  &\mbox{(Repeated applications of $E_2$ in $T$)}\\
T &\models \wedge_{i=1}^n K(\tau_i)
  &\mbox{($T$ is $K$-closed)}\\
T &\models K\phi^s
  &\mbox{(Modus Ponens)}\\
\mathscr N_T &\models KK\phi[s].
  &\mbox{(Definition \ref{defnofintendedmodel})}
\end{align*}
This shows $\mathscr N_T\models E_4$.

Because of the lack of Assigned Validity, showing $\M_T\models E_1$ is tricky.
We indicate a rough sketch.
Carlson's Lemmas 5.23 and 7.1 \cite{carlson2000} (pp.~69 \& 72)
imply $T\models (\mbox{Assigned Validity})$
(we invoke Lemma 7.1 with $\mathscr Q$ a singleton).
As written, Lemma 5.23 demands $T$ also contain $E_3$, but it can be shown this is unnecessary.
Thus we may assume $T$ contains Assigned Validity.
By Lemmas \ref{eoneandassignedvalidityisupgeneric} and \ref{upgenericimpliesgeneric},
$\mathscr N_T\models E_1$.
\end{proof}

Lemma \ref{everythingworksgivenetwo}
explains why weakening $E_2$ to $E'_2$ required two other
seemingly-unrelated weakenings:  adding Assigned Validity, and removing 
$E_4$ altogether.

\begin{lemma}
\label{smtisreupgeneric}
The Mechanicalness schema,
\[
\ucl{\exists e\forall x(K\phi\leftrightarrow x\in W_e)}\,\,\,\,\mbox{($e\not\in \FV(\phi)$)},
\]
is r.e.-upgeneric.
\end{lemma}

\begin{proof}
Let $T$ be any r.e.~$\LEA$-theory containing the Mechanicalness schema.
Let $\bullet^+$ be a stratifier and let $\alpha$ be such that $(K\phi)^+\equiv K^\alpha\phi^+$.
We must show
\[
\M_{T^\oplus}\models \ucl{\exists e\forall x(K^\alpha\phi^+\leftrightarrow x\in W_e)}.
\]
Let $s$ be any assignment and note
\begin{align*}
\{q\in\N\,:\,\M_{T^\oplus}\models K^\alpha\phi^+[s(x|q)]\}
&=
\{q\in\N\,:\,T^\oplus\cap\alpha\models (\phi^+)^{s(x|q)}\}.
  &\mbox{(Definition \ref{stratifiedmodel})}
\end{align*}
By the Church--Turing Thesis, the latter set is r.e., so there is some $p\in\N$ such that
\[
W_p = \{q\in\N\,:\,\M_{T^\oplus}\models K^\alpha\phi^+[s(x|q)]\}.
\]
For all $q\in\N$, the following biconditionals are equivalent:
\begin{align*}
\M_{T^\oplus}\models K^\alpha\phi^+ &\leftrightarrow x\in W_e[s(e|p)(x|q)]\\
\M_{T^\oplus}\models K^\alpha\phi^+[s(e|p)(x|q)]
&\mbox{ iff } \M_{T^\oplus}\models x\in W_e[s(e|p)(x|q)]
  &\mbox{(First-order semantics of $\leftrightarrow$)}\\
\M_{T^\oplus}\models K^\alpha\phi^+[s(x|q)]
&\mbox{ iff } \M_{T^\oplus}\models x\in W_e[s(e|p)(x|q)]
  &\mbox{(Since $e\not\in\FV(\phi)$)}\\
\M_{T^\oplus}\models K^\alpha\phi^+[s(x|q)]
&\mbox{ iff } q\in W_p.
  &\mbox{(Since $\M_{T^\oplus}$ has standard first-order part)}
\end{align*}
The latter is true by definition of $p$.
By arbitrariness of $q$, $\M_{T^\oplus}\models \exists e\forall x(K^\alpha\phi^+\leftrightarrow x\in 
W_e)[s]$.
\end{proof}

\begin{corollary}
\label{evenweakerweak}
(Recall the definition of $T^w_{\text{SMT}}$ from the end of Section \ref{prelimsect})
Let $(T^w_{\text{SMT}})\backslash E_3$ be the smallest $K$-closed theory containing $E_1$,
Assigned Validity, $E'_2$, Epistemic Arithmetic, and Mechanicalness.
(Loosely speaking, $T^w_{\text{SMT}}$ minus $E_3$.)  Then $(T^w_{\text{SMT}})\backslash E_3$
is r.e.-upgeneric.
\end{corollary}

\section{The Main Result}
\label{mainresultsect}

With the machinery of Section \ref{genericaxiomssectn},
we are able to state our main result in a generalized form.
Informally:
\begin{quote}
An r.e.-upgeneric theory
remains true
upon augmentation by knowledge of its own truthfulness.
\end{quote}
Reinhardt's conjecture (proved by Carlson) was that the Strong Mechanistic Thesis is consistent with
a particular background theory of knowledge.
We showed (Lemma \ref{smtisreupgeneric}) that Mechanicalness is r.e.-upgeneric.
By Lemma \ref{kclosurelemma}, the pair consisting of Mechanicalness and the Strong Mechanistic Thesis,
is r.e.-upgeneric.
Thus as long as
the background theory of knowledge
is r.e.~and built of r.e.-generic pieces along with truthfulness,
the corresponding conjecture is a special case of this main result.

Recall
(Definition \ref{defnofintendedmodel}) that an $\LEA$-theory $T$ is \emph{true} if $\mathscr N_T\models T$.

\begin{theorem}
\label{themaintheorem}
Let $T_0$ be an r.e.-upgeneric $\LEA$-theory.
Let $T_1$ be $T_0\cup E_3$, that is, $T_0$ along with all axioms of the form $\ucl{K\phi\rightarrow\phi}$.
Let $T$ be the smallest $K$-closed theory containing $T_1$.
Then $T$ is true.
\end{theorem}

\begin{proof}
By Corollary \ref{revelatorycorollary} it is enough to show
$\M_{T^\oplus}\models T^\oplus$.
We will use transfinite induction up to $\epom$
to show that for all $\alpha\in\epom$, $\M_{T^\oplus}\models T^\oplus\cap\alpha$.

Let $\sigma\in T^\oplus\cap\alpha$.  Then $\sigma\equiv\theta^+$ for some $\theta\in T$ and some stratifier $\bullet^+$.
We will show $\M_{T^\oplus}\models\theta^+$.

\item
\case1
$\theta\in T_0$.
Then $\M_{T^\oplus}\models \theta^+$ because $T\supseteq T_0$ is r.e.~and $T_0$ is r.e.-upgeneric.

\item
\case2
$\theta$ is $K\phi$ for some sentence $\phi\in T$.  Let $\alpha_0$ be such that $(K\phi)^+\equiv K^{\alpha_0}\phi^+$.
By Definition \ref{stratifierdefn}, $\onset(\phi^+)\subseteq \alpha_0$
and thus $\phi^+\in T^\oplus \cap\alpha_0$, so $T^\oplus\cap\alpha_0\models \phi^+$,
so $\M_{T^\oplus}\models K^{\alpha_0}\phi^+$.

\item
\case3
$\theta$ is $\ucl{K\phi\rightarrow\phi}$ for some $\phi$.
Let $\alpha_0$ be such that $(K\phi)^+\equiv K^{\alpha_0}\phi^+$,
so $\theta^+$ is $\ucl{K^{\alpha_0}\phi^+\rightarrow\phi^+}$.
Since $\theta^+\in T^\oplus\cap\alpha$, this forces $\alpha_0<\alpha$.
Let $s$ be any assignment and assume $\M_{T^\oplus}\models K^{\alpha_0}\phi^+[s]$.
Then:
\begin{align*}
\M_{T^\oplus} &\models K^{\alpha_0}\phi^+[s]
  &\mbox{(Assumption)}\\
T^\oplus\cap\alpha_0 &\models (\phi^+)^s
  &\mbox{(Definition \ref{stratifiedmodel})}\\
\M_{T^\oplus} &\models (\phi^+)^s
  &\mbox{(By $\epom$-induction, $\M_{T^\oplus}\models T^\oplus\cap\alpha_0$)}\\
\M_{T^\oplus} &\models \phi^+[s]\mbox{, as desired.}
  &\mbox{(Lemma \ref{scriptmbehavesasintended})} 
\end{align*}
\end{proof}

\begin{corollary}
$T^w_{\text{SMT}}$ is true.
\end{corollary}

\begin{proof}
By Theorem \ref{themaintheorem} and Corollary \ref{evenweakerweak}.
\end{proof}

If one is willing to induct up to $\epsilon_0\cdot\omega$ and
use machinery from \cite{carlson1999}, it is possible (without the grievous sacrifices
we have made in this paper) to generalize Reinhardt's conjecture
to a statement of the form:
\begin{quote}Any r.e.~theory that is generic in a very specific sense
(one that allows $E_2$ as building block)
remains true upon augmentation by knowledge of its own truthfulness. ($*$)\end{quote}
The specific notion of ``generic'' in order for this to work is somewhat complicated
and hinges on \cite{carlson1999}, putting it out of the present paper's scope.
It does admit Mechanicalness as building block,
so that ($*$) really is a generalization of Reinhardt's conjecture,
and the notion also admits full $E_2$, which in turn allows building blocks
containing $E_4$.

The main result of \cite{alexandercode} can also be generalized in this manner.
The methods of that paper are easily modified to prove:
\begin{quote}
For any r.e.~$\LEA$-theory $T$ that is generic (in the sense of Definition \ref{baregenericdefn}),
there is an $n\in\N$ such that $T'$ is true, where $T'$ is the smallest $K$-closed theory
containing $T$ along with the schema $\forall x(K\phi\leftrightarrow \langle 
x,\overline{\ulcorner\phi\urcorner}\rangle\in W_{\overline n})$ ($\FV(\phi)\subseteq\{x\}$).
Less formally, any such generic knowing machine can be taught its own code and still remain true.
\end{quote}

One possible application of this paper is to reverse mathematics \cite{simpson}.
Since the results (except Lemma \ref{eaisupgeneric}) only use induction
up to $\epom$,
suitable versions (minus Lemma \ref{eaisupgeneric} and
references to $\N$)
could be formalized and proved in weak subsystems of arithmetic.

\end{document}